\documentclass[11pt,reqno]{amsart}
\usepackage{amsbsy,amsfonts,amsmath,amssymb,amscd,amsthm,mathrsfs}
\usepackage{subfigure,enumitem,graphicx,epsfig,color}
\usepackage[a4paper,text={6in,8in},centering,includefoot,foot=0.6in]{geometry}
\normalsize

\newcommand{\st}{\stackrel}

\newcommand{\T}{\mathrm}

\newtheorem{theorem}{Theorem}[section]
\newtheorem{corollary}[theorem]{Corollary}
\newtheorem{lemma}[theorem]{Lemma}
\newtheorem{proposition}[theorem]{Proposition}
\newtheorem{definition}[theorem]{Definition}
\newtheorem{definitions}[theorem]{Definitions}
\newtheorem{examples}[theorem]{Examples}
\newtheorem{example}[theorem]{Example}
\newtheorem{notation}[theorem]{Notation}
\newtheorem{remark}[theorem]{Remark}
\newtheorem{remarks}[theorem]{Remarks}

\numberwithin{equation}{section}

\begin{document}
\bibliographystyle{amsplain}

\title[algebraic invariants of edge ideals of hypergraphs]{Some interpretations for algebraic invariants of edge ideals of hypergraphs via combinatorial invariants}
\author[S. Moradi and F. Khosh-Ahang]{Somayeh Moradi and Fahimeh Khosh-Ahang$^*$ }

\address{Somayeh Moradi, Department of Mathematics, School of Science, Ilam University, P.O.Box 69315-516, Ilam, Iran.} 
\email{so.moradi@ilam.ac.ir}
\address{Fahimeh Khosh-Ahang, Department of Mathematics, School of Science, Ilam University, P.O.Box 69315-516, Ilam, Iran.}
\email{fahime$_{-}$khosh@yahoo.com, f.khoshahang@ilam.ac.ir}

\keywords{graded Betti numbers, projective dimension, regularity, triangulated hypergraph.  \\
$*$Corresponding author}
\subjclass[2010]{Primary 13F55, 13D02, 05C90;  Secondary 05E99}

\begin{abstract}
The present work is concerned with characterizing some algebraic invariants of edge ideals of hypergraphs. To this aim, firstly, we introduce some kinds of combinatorial invariants similar to matching numbers for hypergraphs. Then we compare them to each other and to previously existing ones. These invariants are used for characterizing or bounding some algebraic invariants of edge ideals of hypergraphs such as  graded Betti numbers, projective dimension and Castelnouvo-Mumford regularity.
\end{abstract}

\maketitle

\section*{Introduction}

Throughout this paper, suppose that $\mathcal{H}$ is a \textbf{simple hypergraph} on the vertex set $V(\mathcal{H})=\{x_1, \dots , x_n\}$ with the edge set $\mathcal{E}(\mathcal{H})$. That is for each $S\in \mathcal{E}(\mathcal{H})$, $|S|\geq 2$ and for every distinct edges $S$ and $S'$ of $\mathcal{H}$, $S\nsubseteq S'$.   Also, for each $W\subseteq V(\mathcal{H})$, the \textbf{induced subhypergraph} $\mathcal{H}_W$ is the hypergraph on the vertex set $W$ whose edges are the edges of $\mathcal{H}$ which are contained in $W$. Moreover, if $S\in \mathcal{E}(\mathcal{H})$, then we use the notation $\mathcal{H}\setminus S$ for the hypergraph with the vertex set $V(\mathcal{H})$ and the edge set $\mathcal{E}(\mathcal{H})\setminus \{S\}$.
We identify the vertex $x_i$ of $\mathcal{H}$ with the variable $x_i$ of the polynomial ring over the field $K$, which we denote by $R: R=K[x_1, \dots , x_n]$. Consider $R$ as $\mathbb{N}$-graded ring by defining $\mathrm{deg}x_i=1$. For each subset $S$ of $V(\mathcal{H})$,  denote the monomial $\prod_{x_i\in S}x_i$ in $R$ by $x^S$. For our convenience, we sometimes denote the subset $\{x_{i_1}, \dots , x_{i_k}\}$ of $\{x_1, \dots , x_n\}$ by the monomial $x_{i_1}\cdots x_{i_k}$. The squarefree monomial ideal
$$I(\mathcal{H})=\langle x^S \ | \ S\in \mathcal{E}(\mathcal{H}) \rangle$$ of $R$ is called the \textbf{edge ideal} of $\mathcal{H}$.
We say that a simple hypergraph $\mathcal{H}$ is \textbf{$d$-uniform} if all edges of $\mathcal{H}$ have the same cardinality $d$. When this is the case, $I(\mathcal{H})$ is a squarefree monomial ideal generated in degree $d$.

For a minimal graded free resolution of $R/I(\mathcal{H})$
$$\cdots \longrightarrow \oplus_jR(-j)^{\beta_{i,j}} \longrightarrow \cdots \longrightarrow \oplus_jR(-j)^{\beta_{1,j}} \longrightarrow R \longrightarrow R/I(\mathcal{H}) \longrightarrow 0,$$
$\beta_{i,j}(R/I(\mathcal{H}))$ is called the $(i,j)$th \textbf{graded Betti number} of $R/I(\mathcal{H})$. Also, recall that the \textbf{Castelnuovo-Mumford regularity} (or simply \textbf{regularity})
of  $R/I(\mathcal{H})$ is defined as
$$\T{reg}(R/I(\mathcal{H})) = \max\{j-i \ | \ \beta_{i,j}(R/I(\mathcal{H}))\neq 0\},$$
and the \textbf{projective dimension} of $R/I(\mathcal{H})$ is defined as
$$\T{pd}(R/I(\mathcal{H})) = \max\{i \ | \ \beta_{i,j}(R/I(\mathcal{H}))\neq 0 \ \text{for some}\ j\}.$$

It is well known that there is a bijection between the set of all squarefree monomial ideals of $R$ and hypergraphs with the vertex set $\{x_1, \dots , x_n\}$ via the edge ideals after the worthy work of Villarreal \cite{Villarreal}. Hence translating algebraic properties of the edge ideal of a hypergraph $\mathcal{H}$ to combinatorial properties of $\mathcal{H}$ has attracted considerable  attention for more than two decades. In general, it is hard to determine the graded Betti numbers, projective dimension and regularity of a monomial ideal or even bounding them. Therefore, it is worth to find some classes of hypergraphs  (resp. graphs) for which these invariants of their edge ideals can be determined or bounded by some combinatorial aspects of the underlying hypergraphs (resp. graphs). In this regard, there are many papers devoted to studying this problem (cf. \cite{Ha+Vantuyl}, \cite{Katzman}, \cite{FS}, \cite{Kimura}, \cite{F+S}, \cite{M+V}, \cite{Zheng} and etc.). Nevertheless most of works are about edge ideals of graphs and so generalizing the gained results on graphs to hypergraphs, to cover all squarefree monomial ideals, makes sense.

The main goal of this paper is to find some relations between the graded Betti numbers of the edge ideal of a hypergraph and some combinatorial invariants associated to the hypergraph. In particular, we seek for some combinatorial descriptions for the projective dimension and regularity of edge ideals.  In this way, the motivation was to generalize some combinatorial characterizations or bounds of regularity and projective dimension of the edge ideals which were presented in \cite{Ha+Vantuyl}, \cite{Katzman}, \cite{Kimura} and \cite{M+V}. To this end, in the first section, we introduce some new combinatorial invariants for hypergraphs and compare them with each other and the  previously existing ones. Then in Section 2, after recalling some preliminaries, we find some bounds for $\beta_{i,j}(R/I(\mathcal{H}))$ in Theorem \ref{3.7} under certain circumstances. This result helps us  to generalize Katzman's argument and to cover \cite[Theorem 6.5]{Ha+Vantuyl}, \cite[Lemma 2.2 and Proposition 2.5]{Katzman} and
\cite[Corollary 3.9]{M+V} in Corollary \ref{C3.7} and Theorem \ref{T3.7}. Also in Theorem \ref{1.9}, we introduce some combinatorial lower bounds for projective dimension and regularity of edge ideal of a simple hypergraph  which covers Theorem 3.1 in \cite{Kimura}. In the third section, we present a precise interpretation for projective dimension and regularity of edge ideal of special class of hypergraphs by some combinatorial invariants  which is a generalization of  some of the main results of \cite{Kimura} and \cite{Zheng} (see Theorem \ref{1.11}).

\section{Some hypergraph invariants}
There are some hypergraph invariants which lead us to characterizing graded Betti numbers and so  finding some bounds for projective dimension and regularity of $R/I(\mathcal{H})$ in the next sections. In this regard, for convenience of the reader, we have gathered together a complete list of these invariants consisting of some previously existing invariants and some new ones in the following definitions.
\begin{definitions} \label{1.2}
Let $\mathcal{S}=\{S_1, \dots , S_i\}$ be a family of edges of $\mathcal{H}$.
\begin{itemize}
\item[1.] (See \cite{Berge}.) $\mathcal{S}$ is called a \textbf{matching} in $\mathcal{H}$ if for each $1\leq \ell < \ell' \leq i$,  $S_\ell\cap S_{\ell '}=\emptyset$.
\item[2.] (See \cite[Definition 2.1]{F+S}.)  $\mathcal{S}$ is called a \textbf{semi-induced matching} in $\mathcal{H}$ if for each $S\in \mathcal{E}(\mathcal{H})\setminus \{S_1, \dots , S_i\}$, $S\nsubseteq \bigcup_{\ell=1}^i S_\ell$ or equivalently the induced subhypergraph on $\bigcup_{\ell=1}^i S_\ell$ has only the edges $S_1, \dots, S_i$.
\item[3.]  $\mathcal{S}$ is called a \textbf{self semi-induced matching} in $\mathcal{H}$ if
\begin{itemize}
\item[i.] $\mathcal{S}$ is a semi-induced matching in $\mathcal{H}$;
\item[ii.] for all $1\leq k \leq i$, $S_k\nsubseteq \bigcup _{1\leq \ell \leq i,\ell\neq k} S_\ell$.
\end{itemize}
\item[4.] $\mathcal{S}$ is called a \textbf{self-contained semi-induced matching} in $\mathcal{H}$ if
\begin{itemize}
\item[i.] for each $S\in \mathcal{E}(\mathcal{H})\setminus \{S_1, \dots , S_i\}$, $S\nsubseteq \bigcup_{\ell=1}^i S_\ell$ or there exists $1\leq k \leq i$ such that $S_k\subseteq S\cup (\bigcup _{1\leq \ell\leq i,\ell\neq k} S_\ell)$;
\item[ii.] for all $1\leq k \leq i$, $S_k\nsubseteq \bigcup _{1\leq \ell\leq i,\ell\neq k} S_\ell$.
\end{itemize}
\item[5.] (See \cite{M+V}.) $\mathcal{S}$ is called an \textbf{induced matching} in $\mathcal{H}$ if
\begin{itemize}
\item[i.] $\mathcal{S}$ is a semi-induced matching in $\mathcal{H}$;
\item[ii.] $\mathcal{S}$ is a matching in $\mathcal{H}$.
\end{itemize}
\item[6.] $\mathcal{S}$ is called a \textbf{ self disjoint set (self semi-disjoint set)}  in $\mathcal{H}$ if
\begin{itemize}
\item[i.] for all $1\leq k \leq i$, $S_k\nsubseteq \bigcup _{1\leq \ell \leq i,\ell\neq k} S_\ell$;
\item[ii.] there is an  induced matching (a semi-induced matching) $\mathcal{S}_0=\{S_{k_1}, \dots , S_{k_t}\}$ contained in  $\mathcal{S}$ such that for each $\ell\in\{1, \dots , i\}\setminus \{k_1, \dots , k_t\}$, there exists $1\leq \ell ' \leq t$ with $|S_\ell\setminus S_{k_{\ell '}}|=1$.
\end{itemize}
\item[7.] $\mathcal{S}$ is called a \textbf{ self ordered set of edges}  in $\mathcal{H}$ if either $i=1$ or
\begin{itemize}
\item[i.] for all $1\leq k \leq i$, $S_k\nsubseteq \bigcup _{1\leq \ell \leq i,\ell\neq k} S_\ell$;
\item[ii.] for each $S\in \mathcal{E}(\mathcal{H})\setminus \{S_1, \dots , S_i\}$ there exists $k<i$ such that $S_k\subseteq S\cup (\bigcup_{\ell=k+1}^iS_\ell)$. (Note that in this definition the order of edges in $\mathcal{S}$ is important.)
\end{itemize}
\end{itemize}
Assume that $\mathcal{S}=\{S_1, \dots , S_i\}$ is one of the above ones and set $j=|\bigcup_{\ell =1}^i S_\ell |$. Then we define the type of $\mathcal{S}$ as $(i,j)$.
\end{definitions}
Now, we recall the following definition from \cite{Kimura} and \cite{Zheng}.
\begin{definition}\label{1.2.1}
(See \cite{Kimura} and \cite{Zheng}.) A graph $B$ with the vertex set $\{x, y_1,\dots, y_t\}$, $t\geq 1$, and the edges $\{x, y_\ell\}$
for $\ell = 1, \dots , t$ is called a \textbf{bouquet}. The
vertex $x$ is called the \textbf{root}, the vertices $y_\ell$ the \textbf{flowers} and the edges $\{x, y_\ell\}$ the \textbf{stems}
of this bouquet. A subgraph $B$ of $G$ which is a bouquet is called a bouquet of $G$.

A set $\mathcal{B}=\{B_1, \dots , B_j\}$ of bouquets of $G$ is called \textbf{strongly disjoint set of bouquets} if
\begin{itemize}
\item[i.] for all $k\neq \ell$, $V(B_k)\cap V(B_\ell)=\emptyset$;
\item[ii.] we can choose a stem $S_k$ from each bouquet $B_k$ so that $\{S_1, \dots , S_j\}$ is an induced matching in $G$.
\end{itemize}
When this is the case, if the number of flowers of $B_k$s is $i$, then we say that $\mathcal{B}$ is of type $(i,j)$.
\end{definition}

Throughout this paper, we also need the following notation.
\begin{notation}
For a hypergraph $\mathcal{H}$ we use the following notation.
\begin{align*}
&m_{\mathcal{H}}=\max \{i \ | \  \textrm{there is a } \textrm{matching of size  } i \textrm{ in } \mathcal{H} \},\\
&a_{\mathcal{H}}=\max \{i \ | \  \textrm{there is an induced matching of size  } i \textrm{ in } \mathcal{H} \},\\
&a_{\mathcal{H},t}=\max \{i \ | \  \textrm{there is an induced matching of size  } i \textrm{ in } \mathcal{H} \textrm{ consisting of } t\textrm{-sets}\},\\
&b_{\mathcal{H}}=\max \{i \ | \  \textrm{there is a self semi-induced matching of size  } i \textrm{ in } \mathcal{H}  \},\\
&b'_{\mathcal{H}}=\max \{j-i \ | \  \textrm{there is a self semi-induced matching of type  }(i,j) \textrm{ in } \mathcal{H}  \},\\
&c_{\mathcal{H}}=\max \{i \ | \ \textrm{there is a self ordered set of edges of size  } i \textrm{ in } \mathcal{H}  \},\\
&c'_{\mathcal{H}}=\max \{j-i \ | \ \textrm{there is a self ordered set of edges in }  \mathcal{H} \textrm{ of type } (i,j) \},\\
&d_{1,\mathcal{H}}=\max \{i \ | \ \textrm{there is a  self disjoint set in }  \mathcal{H} \textrm{ of size } i \},\\
&d_{2,\mathcal{H}}=\max \{i \ | \ \textrm{there is a self semi-disjoint set in }  \mathcal{H} \textrm{ of size } i \},\\
&d'_{1,\mathcal{H}}=\max \{j-i \ | \ \textrm{there is a self disjoint set in }  \mathcal{H} \textrm{ of type } (i,j) \},\\
&d'_{2,\mathcal{H}}=\max \{j-i \ | \ \textrm{there is a self semi-disjoint set in }  \mathcal{H} \textrm{ of type } (i,j) \},\\
&d_G=\max\{i \ | \  \textrm{there is a strongly disjoint set of bouquets in }  G \textrm{ of type } (i,j) \},\\
&d'_G=\max\{j \ | \  \textrm{there is a strongly disjoint set of bouquets in }  G \textrm{ of type } (i,j) \},\\
&e_{\mathcal{H}}=\max \{i \ | \ \textrm{there is a self-contained semi-induced matching of size  }i \textrm{ in } \mathcal{H}  \}.
\end{align*}
\end{notation}

In previous notation,  the invariants $m_{\mathcal{H}}$ and $a_{\mathcal{H}}$  are known as the matching number and the induced matching number in hypergraphs and $d_G$ is as defined in \cite{Kimura}. Also it can be easily seen that if $\mathcal{H}$ is a $d$-uniform hypergraph, then $a_{\mathcal{H},d}=a_{\mathcal{H}}$. Moreover, H$\T{\grave{a}}$ and Van Tuyl in \cite{Ha+Vantuyl} introduced the concepts of properly-connected hypergraphs and pairwise $t$-disjoint sets of edges in a $d$-uniform properly-connected hypergraph. As one can see in the proof of Theorem 6.5 in \cite{Ha+Vantuyl}, the authors have shown that in a $d$-uniform properly-connected hypergraph, a set of edges in $\mathcal{H}$ is an induced matching if and only if it is a pairwise $(d+1)$-disjoint set.

In the following remarks, we compare the defined concepts in Definitions \ref{1.2} with each other.

\begin{remarks}\label{remarks1.3}
\begin{itemize}
\item[1.] In the light of Definitions \ref{1.2}, we have the following implications.
$$\begin{array}{lllll}
  &   & \textrm{matching} &   & \textrm{semi-induced matching} \\
  & \nearrow &   & \nearrow &   \\
\textrm{induced matching} & \longrightarrow & \textrm{self semi-induced matching} & \longrightarrow & \textrm{self-contained semi-induced matching} \\
  & \searrow &   & \searrow &   \\
  &   & \textrm{self disjoint set} & \longrightarrow & \textrm{self semi-disjoint set}\\
\end{array} 
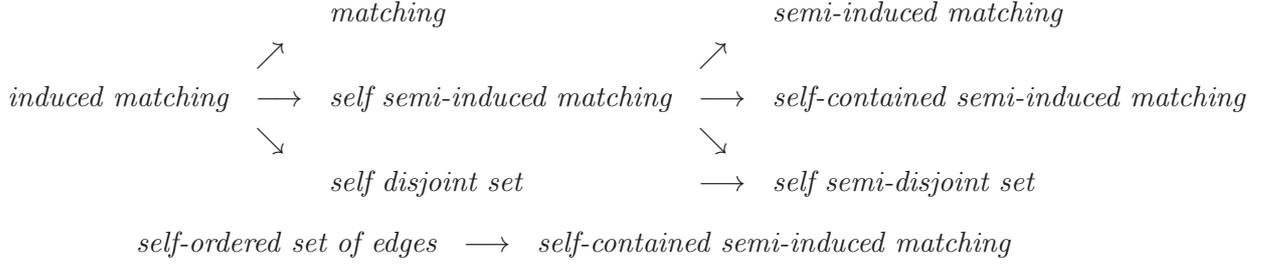
$$
$$\begin{array}{lll}
\textrm{self-ordered set of edges} & \longrightarrow & \textrm{self-contained semi-induced matching}
\end{array} $$

\item[2.] In view of Part 1, we have the following inequalities.
$$a_{\mathcal{H}}\leq m_{\mathcal{H}},$$
$$a_{\mathcal{H}}\leq b_{\mathcal{H}}\leq \min\{d_{2,\mathcal{H}}, e_{\mathcal{H}}\}, $$
$$a_{\mathcal{H}} \leq  d_{1,\mathcal{H}}\leq  d_{2,\mathcal{H}},$$
$$c_{\mathcal{H}}\leq e_{\mathcal{H}},$$
$$b'_{\mathcal{H}}\leq d'_{2,\mathcal{H}},$$
$$d'_{1,\mathcal{H}}\leq d'_{2,\mathcal{H}}.$$
\end{itemize}
\end{remarks}

The following proposition illustrates that when $G$ is a graph, the invariants $d_{1, G}$, $d_{2,G}$ and $d_G$ coincide. So, as the reader will see in Theorems \ref{1.9} and \ref{1.11}, our invariants $d_{1, \mathcal{H}}$ and $d_{2,\mathcal{H}}$ may be some efficient generalizations of $d_G$, which is defined in \cite{Kimura}, for hypergraphs.
\begin{proposition}\label{proposition1.3}
\begin{itemize}
\item[1.] If $G$ is a graph and  $\mathcal{B}$ is a strongly disjoint  set of bouquets of $G$ of type $(i,j)$, then $\mathcal{E}(\mathcal{B})$ is a self disjoint set in $G$ of type $(i, i+j)$. Conversely, if $\mathcal{S}$ is a self disjoint set in $G$ of type $(i,i+j)$, then it can be seen that $\mathcal{S}=\mathcal{E}(\mathcal{B})$ for some $\mathcal{B}$ which is a strongly disjoint set of bouquets of $G$ of type $(i,j)$.
\item[2.] Every self semi-disjoint set in a graph $G$ is a self disjoint set.
\item[3.] For any graph $G$, we have $d_{1, G}=d_{2,G}=d_G$ and $d'_{1, G}=d'_{2,G}=d'_G$.
\end{itemize}
\end{proposition}
\begin{proof}
\begin{itemize}
\item[1.] The first statement is straightforward consequence of Definitions \ref{1.2} and \ref{1.2.1}. Conversely if $\mathcal{S}$ is a self disjoint set in $G$ of type $(i,i+j)$, then Condition (i) in Definitions \ref{1.2}(6) ensures that $\mathcal{S}=\mathcal{E}(\mathcal{B})$ for some $\mathcal{B}$ which is a set of bouquets in $G$.  Also, Condition (ii) in Definitions \ref{1.2}(6) implies that $\mathcal{S}_0$ has at least one stem from each bouquet of $\mathcal{B}$. Since $\mathcal{S}_0$ is an induced matching, $\mathcal{S}_0$ has exactly one stem from each bouquet of $\mathcal{B}$. Therefore, $\mathcal{B}$ is a strongly disjoint set of bouquets of $G$ of type $(i,j)$.
\item[2.] Note that in a graph $G$, if $\mathcal{S}$ is a self semi-disjoint set, then Condition (i) in Definitions \ref{1.2}(6) ensures that $\mathcal{S}=\mathcal{E}(\mathcal{B})$ for some $\mathcal{B}$ which is a set of bouquets in $G$. Also, Condition (ii) in Definitions \ref{1.2}(6) implies that $\mathcal{S}_0$ has at least one stem from each bouquet of $\mathcal{B}$. Hence if we choose one stem from each bouquet of $\mathcal{B}$ which lies in $\mathcal{S}_0$, then it makes $\mathcal{S}$ into a self disjoint set.
\item[3.] As we know, the notion of self semi-disjoint set and self disjoint set are equivalent for the graph by Part 2 and Remarks \ref{remarks1.3}(1). Hence by Part 1 the equalities can be gained.
\end{itemize}
\end{proof}

The following examples show that even in graphs, the inverse implications in Remarks \ref{remarks1.3}(1) do not necessarily hold and the inequalities in Remarks \ref{remarks1.3}(2) can be strict.
\begin{examples}
\begin{itemize}
\item[1.] Assume that $G$ is a cycle on the vertex set $\{x,y,z\}$. Then it can be easily seen that $\{xy, xz\}$ is a maximal self-contained semi-induced matching which is not a self semi-induced matching. We have $b_G=1< 2= e_G$.

\item[2.] Suppose that $G$ is a path with $V(G)=\{x,y,z\}$ and $E(G)=\{xy, yz\}$. Then it can be seen that $\{xy, yz\}$ is a maximal self semi-induced matching which is not an induced matching. We have $a_G=1 <2=b_G$.

\item[3.] Assume that $G$ is a path with $V(G)=\{w, x, y, z\}$ and $E(G)=\{wx,xy, yz\}$. Then one can easily check that $\{wx, yz\}$ is a matching which is not an induced matching and we have $a_G=1<2=m_G$.

\item[4.] Let $G$ be a path with $V(G)=\{u, v, w, x, y, z\}$ and $E(G)=\{uv, vw, wx, xy, yz\}$.
Then one can check that $\{uv, vw, xy, yz\}$ is a self disjoint set of edges in $G$ which is not self semi-induced matching. We have $b_G=3<4=d_G$.

\item[5.] Let $G$ be the cycle with $V(G)=\{w, x, y, z\}$ and $E(G)=\{wx, xy, yz, zw\}$. Then one can easily check that $\{wx, xy\}$ is a maximal self-contained semi-induced matching in $G$ which is not a self-ordered set of edges. Hence $e_G=2>1=c_G$.

\item[6.] Assume that $\mathcal{H}$ is the hypergraph with $V(\mathcal{H})=\{x_1, \dots, x_6\}$ and
$$\mathcal{E}(\mathcal{H})=\{x_1x_2x_3, x_2x_3x_4,x_2x_5x_6\}.$$
 Then it can be easily check that $\mathcal{E}(\mathcal{H})$ is a self semi-disjoint set which is not a self disjoint set and so $d_{1,\mathcal{H}}<d_{2,\mathcal{H}}$.
\end{itemize}
\end{examples}

\section{Graded Betti numbers and hypergraph invariants}
We begin this section by the following remarks which all of its parts are trivial facts or straightforward consequences of Hochster's formula (\cite[Theorem 5.1]{Hoch}.

\begin{remarks}(Compare \cite[Corollary 1.2]{Katzman}, \cite[Lemma 3.3]{Kimura} and \cite[Proposition 3.8]{M+V}.)\label{1.3}
For every hypergraph $\mathcal{H}$ with $n$ vertices and every integers $i$ and $j$, the following statements hold.
\begin{itemize}
\item[1.] $\beta_{i,j}(R/I(\mathcal{H}))=\sum_{W\subseteq V(\mathcal{H}), |W|=j}\beta_{i,j}(R/I(\mathcal{H}_W))$.
\item[2.] Part 1 shows that if  $\mathcal{H}'$ is an induced subhypergraph of $\mathcal{H}$, then
$$\beta_{i,j}(R/I(\mathcal{H}'))\leq \beta_{i,j}(R/I(\mathcal{H})),$$
and so
$$\mathrm{pd}(R/I(\mathcal{H}'))\leq \mathrm{pd}(R/I(\mathcal{H})),$$
and
$$\mathrm{reg}(R/I(\mathcal{H}'))\leq \mathrm{reg}(R/I(\mathcal{H})).$$
\item[3.] Set $t=\max \{|S| \ | \ S\in \mathcal{E}(\mathcal{H})\}$ and $t'=\min \{|S| \ | \ S\in \mathcal{E}(\mathcal{H})\}$. Then $\beta_{i,j}(R/I(\mathcal{H}))\neq 0$ implies $i+t'-1\leq j \leq \min\{n,ti\}$ by the fact that $\beta_{i,j}(R/I(\mathcal{H}))=0$ for all $i$ and $j$ with $j<i+t'-1$, and Part 1.
\item[4.] Part 3 ensures that $t'-1\leq \mathrm{reg}(R/I(\mathcal{H})) \leq (t-1)\mathrm{pd}(R/I(\mathcal{H}))$.
\end{itemize}
\end{remarks}

Now we are going to explain the Taylor resolution and a Lyubeznik resolution of $R/I(\mathcal{H})$ and to use their notation hereafter. Let $\mathcal{E}(\mathcal{H})=\{S_1, \dots , S_m\}$ and $I(\mathcal{H})=\langle x^{S_1}, \dots , x^{S_m}\rangle$. Let $T_0=R$ and $T_i$ be the free $R$-module whose free generators are $e_{\ell_1, \dots , \ell_i}$, where $1\leq \ell_1< \dots < \ell_i\leq m$. For each $i\geq 1$ and $e_{\ell_1, \dots , \ell_i}\in T_i$ define
$\partial _i:T_i \rightarrow T_{i-1}$ with
$$\partial_i(e_{\ell_1, \dots , \ell_i})=\sum_{k=1}^i (-1)^k\mu_k e_{\ell_1, \dots, \widehat{\ell_k}, \dots,  \ell_i},$$
where $\mu_k=x^{S_{\ell_k}\setminus (\cup _{1\leq t \leq i,t\neq k} S_{\ell_t})}$. Then
$$T_\bullet : \dots \rightarrow T_{i+1} \st{\partial_{i+1}}{\longrightarrow} T_i \st{\partial_i}{\longrightarrow} T_{i-1}\rightarrow  \dots \rightarrow T_0 \rightarrow R/I(\mathcal{H}) \rightarrow 0$$
is a free resolution of $R/I(\mathcal{H})$ which is called the \textbf{Taylor resolution} of $R/I(\mathcal{H})$. Considering the degree of
$e_{\ell_1, \dots,  \ell_i}$ as
$$\mathrm{deg}(e_{\ell_1, \dots,  \ell_i})=\mathrm{deg} (\mathrm{lcm}(x^{S_{\ell_1}}, \dots , x^{S_{\ell_i}}))=|\bigcup_{t=1}^iS_{\ell_t}|,$$
 we have that $T_\bullet$ is a graded free resolution of $R/I(\mathcal{H})$ which is not necessarily minimal. But we may use it for computing the graded Betti numbers $\beta _{i,j}(R/I(\mathcal{H}))$ as follows.
\begin{equation}\label{equation3}
\begin{aligned}
\beta _{i,j}(R/I(\mathcal{H}))&=\mathrm{dim}_K(\mathrm{Tor}_{i}^R(R/I(\mathcal{H}),K))_j\\
&=\mathrm{dim}_K(H_{i}(T_\bullet \otimes _R  R/\langle x_1, \dots , x_n\rangle ))_j\\
&=\mathrm{dim}_K(\mathrm{Ker}\overline{\partial}_{i}/\mathrm{Im}\overline{\partial}_{i+1})_j.
\end{aligned}
\end{equation}
One can check that after tensoring $T_\bullet$ with $R/\langle x_1, \dots , x_n\rangle$, we have
\begin{equation}\label{equation4}
\overline{\partial}_i (\overline{e_{\ell_1, \dots ,  \ell_i}})=\sum_{S_{\ell_k}\subseteq \cup _{1\leq t \leq i,t\neq k} S_{\ell_t}} (-1)^k \overline{e_{\ell_1, \dots, \widehat{\ell_k}, \dots,  \ell_i}},
\end{equation}
where for each $0\leq i \leq m$ and each member $u\in T_i$, $\overline{u}$ is the natural image of $u$ in $\overline{T_i}=T_i\otimes_R R/\langle x_1, \dots , x_n\rangle$ and $\overline{\partial}_i =\partial_i\otimes_R \mathrm{id}_{R/\langle x_1, \dots , x_n\rangle }$.

Now, consider an ordering on edges of $\mathcal{H}$. The free generator  $e_{\ell_1, \dots ,  \ell_i}$  is called an\textbf{ L-admissible symbol} if for all $t<i$ and all $q<\ell_t$, $S_q\nsubseteq \bigcup_{k=\ell_t}^{\ell_i} S_k$. An L-admissible symbol
$e_{\ell_1, \dots ,\ell_i}$ is said to be a \textbf{maximal L-admissible symbol} if there is no another L-admissible symbol  $e_{k_1, \dots, k_t}$ such that $\{\ell_1, \dots , \ell_i\} \subseteq \{k_1, \dots, k_t\}$ (see \cite{Barile}). A \textbf{Lyubeznik resolution} of $R/I(\mathcal{H})$ is a subcomplex of the Taylor resolution constructed as follows
$$L_\bullet : \dots \rightarrow L_{i+1} \st{\sigma_{i+1}}{\longrightarrow} L_i \st{\sigma_i}{\longrightarrow} L_{i-1}\rightarrow  \dots \rightarrow  L_0 \rightarrow R/I(\mathcal{H}) \rightarrow 0,$$
 where $L_0=R$ and for each integer $i>0$, $L_i$ is the free $R$-module whose free generators are all L-admissible symbols $e_{\ell_1, \dots , \ell_i}$  (see \cite{Lyubeznik}). Also, for each $i\geq 1$ and $e_{\ell_1, \dots , \ell_i}\in L_i$,
$$\sigma_i(e_{\ell_1, \dots , \ell_i})=\sum_{k=1}^i(-1)^{k}\mu_ke_{\ell_1, \dots, \widehat{\ell_k}, \dots,  \ell_i},$$
where $\mu_k$ is as in the Taylor resolution.  A Lyubeznik resolution also gives a free resolution for $R/I(\mathcal{H})$ which is not necessarily minimal but similar to above discussions, it can be also used for computing the graded Betti numbers of $R/I(\mathcal{H})$. Note that a Lyubeznik resolution of $R/I(\mathcal{H})$ depends on the order which is considered on the edges of $\mathcal{H}$.

The following lemma plays a key role in the sequel.

\begin{lemma}\label{Lemma2.2}
Let $i,j$ be integers. Set
$$\mathcal{B}_{i,j}=\{\overline{e_{\ell_1, \dots , \ell_i}} \ | \ \overline{e_{\ell_1, \dots , \ell_i}}\in \mathrm{Ker}\overline{\partial}_i\setminus \mathrm{Im}\overline{\partial}_{i+1}, |\bigcup_{k=1}^iS_{\ell_k}|=j\}.$$
\begin{itemize}
\item[1.] If for each $\{S_1, \dots , S_i\}\subseteq \mathcal{E}(\mathcal{H})$ with $|\bigcup_{\ell=1}^iS_\ell |=j$ such that $S_k$s are pairwise distinct, we have $S_1\nsubseteq \bigcup_{\ell=2}^iS_\ell$, then
$\{e+\mathrm{Im}\overline{\partial}_{i+1} \ | \ e\in \mathcal{B}_{i,j}\}$ is a generating set for $(\mathrm{Ker}\overline{\partial}_i/\mathrm{Im}\overline{\partial}_{i+1})_j$ over $K$
 and so
$$\beta_{i,j}(R/I(\mathcal{H}))\leq |\mathcal{B}_{i,j}|.$$

\item[2.]  If for each $\{S_1, \dots , S_{i+1}\}\subseteq \mathcal{E}(\mathcal{H})$ with $S_{i+1}\subseteq \bigcup_{\ell =1}^i S_\ell$ and $|\bigcup_{\ell=1}^iS_\ell |=j$ such that $S_k$s are pairwise distinct,  we have $S_k\nsubseteq \bigcup_{1\leq \ell \leq i+1,\ell\neq k}S_\ell$ for all $1\leq k \leq i$, then
$\{e+\mathrm{Im}\overline{\partial}_{i+1}\ | \ e\in \mathcal{B}_{i,j}\}$ is a linear independent subset of $(\mathrm{Ker}\overline{\partial}_i/\mathrm{Im}\overline{\partial}_{i+1})_j$  over $K$ and so
$$\beta_{i,j}(R/I(\mathcal{H}))\geq |\mathcal{B}_{i,j}|.$$

\item[3.] Suppose that $\mathcal{H}$, $i$ and $j$ satisfy the assumptions of Parts 1 and 2. Then
$$\beta_{i,j}(R/I(\mathcal{H}))=|\mathcal{B}_{i,j}|.$$
In particular when $j=ti$ where $t=\max\{|S| \ | \ S\in \mathcal{E}(\mathcal{H})\}$, we have this equality.

\end{itemize}
\end{lemma}
\begin{proof}
\begin{itemize}
\item[1.] By our assumption, it can be easily seen that $(\overline{T_i})_j=(\mathrm{Ker}\overline{\partial}_i)_j$. This immediately implies the result.
\item[2.] Suppose on contrary that there are $\overline{e_{\ell_1, \dots , \ell_i}}$s in $\mathcal{B}_{i,j}$ and non-zero elements
$r_\ell$ in $K$, where $\ell=(\ell_1, \dots , \ell_i)$ such that
$$\sum_\ell r_\ell \overline{e_{\ell_1, \dots , \ell_i}}\in \mathrm{Im}\overline{\partial}_{i+1}.$$
Then there are sequences $\ell '=(\ell_1', \dots , \ell_{i+1}')$ with $1\leq \ell '_1 < \dots < \ell '_{i+1} \leq m$ and elements $r'_{\ell'}$ in $K$ such that 
\begin{align*}
\sum_\ell r_\ell \overline{e_{\ell_1, \dots , \ell_i}}&= \overline{\partial}_{i+1}(\sum_{\ell '} r'_{\ell'} \overline{e_{\ell '_1, \dots , \ell'_{i+1}}})\\
&= \sum_{\ell'} r'_{\ell'} \overline{\partial}_{i+1}(\overline{e_{\ell '_1, \dots , \ell'_{i+1}}})\\
&=\sum_{\ell'}r'_{\ell'}(\sum_{k, S_{\ell'_k}\subseteq \bigcup_{1\leq t \leq i+1,t\neq k} S_{\ell'_t}}(-1)^k\overline{e_{\ell '_1, \dots,   \widehat{\ell '_k},\dots , \ell'_{i+1}}}).
\end{align*}
So each $ \overline{e_{\ell_1, \dots , \ell_i}}$ which exists in the left side should appear in the right side. That is there exists a sequence $1\leq \ell'_1 < \dots < \ell'_{i+1} \leq m$ and an integer $1\leq k \leq i+1$ such that $\{\ell '_1, \dots,   \widehat{\ell '_k},\dots , \ell '_{i+1}\}=\{\ell_1, \dots , \ell_i\}$ with the condition $S_{\ell'_k}\subseteq \bigcup_{1\leq t \leq i+1,t\neq k} S_{\ell '_t}$. Then by our assumption, $ \overline{\partial}_{i+1}(\overline{e_{\ell '_1, \dots , \ell'_{i+1}}})=\overline{e_{\ell_1, \dots , \ell_i}}$, which contradicts to $\overline{e_{\ell_1, \dots , \ell_i}} \in \mathcal{B}_{i,j}$.
\item[3.] follows from Parts 1 and 2. Also, note that if $|\bigcup_{\ell=1}^iS_\ell |=ti$, then $S_\ell$s should be disjoint and hence the assumptions of Parts 1 and 2 hold.
\end{itemize}
\end{proof}

One can easily check that if $i=1,2$ or $j=ti-1$ ($t$ is as defined in Lemma \ref{Lemma2.2}), or $\mathcal{H}$ is a graph which is a set of bouquets, then the assumption of Part 1 of Lemma \ref{Lemma2.2} holds. Although the assumptions of Lemma \ref{Lemma2.2} looks so restrictive specially for graphs, but for hypergraphs with large cardinality of edges, they are not so. For instance, if $\mathcal{H}$ is a $d$-uniform hypergraph in which the intersection of every two edges has at most one element, then for each integers $i\leq d$ and $j$ the assumptions of Part 1 of Lemma \ref{Lemma2.2} holds.

\begin{remarks}\label{remarks2.2}
\begin{itemize}
\item[1.]  By means of the equality (\ref{equation4}), we see that $\overline{e_{\ell_1, \dots ,  \ell_i}}\in \mathrm{Ker}\overline{\partial}_i$ if and only if for all $1\leq k \leq i$, $S_{\ell_k}\nsubseteq \bigcup _{1\leq t\leq i,t\neq k} S_{\ell_t}$. If $\overline{e_{\ell_1, \dots , \ell_i}}\in \mathrm{Im}\overline{\partial}_{i+1}$, then there exists $S\in \mathcal{E}(\mathcal{H})\setminus \{S_{\ell_1}, \dots , S_{\ell_i}\}$ such that $S\subseteq \bigcup_{t=1}^iS_{\ell_t}$. Moreover if there exists  $S\in \mathcal{E}(\mathcal{H})\setminus \{S_{\ell_1}, \dots , S_{\ell_i}\}$ which is contained in $\bigcup _{t =1}^i S_{\ell_t}$ and for each $1\leq k \leq i$, $S_{\ell_k} \nsubseteq S\cup \bigcup_{1\leq t \leq i,t\neq k}S_{\ell_t}$, then $\overline{e_{\ell_1,\dots , \ell_i}}\in \mathrm{Im}\overline{\partial}_{i+1}$.
\item[2.] Similar to Part 1, we see that $\overline{e_{\ell_1, \dots ,  \ell_i}}\in \mathrm{Ker}\overline{\sigma}_i$ if and only if for all $1\leq k \leq i$, $S_{\ell_k}\nsubseteq \bigcup _{1\leq t \leq i,t\neq k} S_{\ell_t}$. Moreover, clearly if $e_{\ell_1, \dots , \ell_i}$ is a maximal L-admissible symbol, then
$\overline{e_{\ell_1, \dots ,\ell_ i}}\notin \mathrm{Im}\overline{\sigma}_{i+1}$.
\item[3.] If $\mathcal{S}$ is a self semi-induced matching in $\mathcal{H}$, then the symbol associated to $\mathcal{S}$ is an L-admissible symbol under any ordering on edges of $\mathcal{H}$ and vise versa.
\item[4.]  If $\mathcal{S}=\{S_{\ell_1}, \dots , S_{\ell_i}\}$ is a self ordered set of edges in $\mathcal{H}$, then $e_{\ell_1,\dots ,\ell_ i}$ is a maximal L-admissible symbol under any ordering of edges of $\mathcal{H}$ in form of $S_{\ell_1}, \dots , S_{\ell_i}, \mathcal{E}(\mathcal{H})\setminus \mathcal{S}$.
\item[5.]  Assume that $\mathcal{S}=\{S_{\ell_1}, \dots , S_{\ell_i}\}$ is a family of edges of $\mathcal{H}$ such that  for all $1\leq k \leq i$, $S_{\ell_k}\nsubseteq \bigcup _{1\leq t \leq i,t\neq k} S_{\ell_t}$. Then the symbol associated to $\mathcal{S}$ is an L-admissible symbol under any ordering of edges of $\mathcal{H}$ in form of  $\mathcal{S}, (\mathcal{E}(\mathcal{H})\setminus \mathcal{S})$ and vise versa.
\item[6.] Let $S$ be a self semi-disjoint set with a semi-induced matching $\mathcal{S}_0$ as in Definitions \ref{1.2}(6). Then Parts 3 and 5 imply that the symbols associated to $(\mathcal{S}\setminus \mathcal{S}_0), \mathcal{S}_0$ are L-admissible symbols with respect to the ordering
$(\mathcal{S}\setminus \mathcal{S}_0), (\mathcal{E}(\mathcal{H})\setminus \mathcal{S}), \mathcal{S}_0$ or
$(\mathcal{S}\setminus \mathcal{S}_0), \mathcal{S}_0, (\mathcal{E}(\mathcal{H})\setminus \mathcal{S})$.
\end{itemize}
\end{remarks}
The following lemma is needed for our next theorem.
\begin{lemma}\label{3.7.1}
Let $i,j$ be integers. Set
$$\mathcal{B}_{i,j}=\{\overline{e_{\ell_1, \dots , \ell_i}} \ | \ \overline{e_{\ell_1, \dots , \ell_i}}\in \mathrm{Ker}\overline{\partial}_i\setminus \mathrm{Im}\overline{\partial}_{i+1}, |\bigcup_{k=1}^iS_{\ell_k}|=j\}.$$
Then
\begin{align*}
&|\{\mathcal{S}\ | \ \mathcal{S} \textrm{ is a self semi-induced matching  in  } \mathcal{H} \textrm{ of type }(i,j)\}| \leq \\
& |\mathcal{B}_{i,j}| \leq \\
&|\{\mathcal{S}\ | \ \mathcal{S} \textrm{ is a self-contained semi-induced matching  in  } \mathcal{H} \textrm{ of type }(i,j)\}|.
\end{align*}
\end{lemma}
\begin{proof}
By means of Part 1 of Remarks \ref{remarks2.2} if $\{S_{\ell_1}, \dots , S_{\ell_i}\}$ is a self semi-induced matching of type $(i,j)$, then $\overline{e_{\ell_1, \dots , \ell_i}}\in (\mathrm{Ker}\overline{\partial}_{i})_j\setminus (\mathrm{Im}\overline{\partial}_{i+1})_j$. Furthermore
$\overline{e_{\ell_1, \dots , \ell_i}}\in (\mathrm{Ker}\overline{\partial}_{i})_j \setminus (\mathrm{Im}\overline{\partial}_{i+1})_j$ implies that $\{S_{\ell_1}, \dots , S_{\ell_i}\}$ is a self-contained semi-induced matching of type $(i,j)$ in $\mathcal{H}$.  These complete the proof.
\end{proof}

Now, we are ready to state one of our main results of this paper.
\begin{theorem}\label{3.7}
For given integers $i$ and $j$ the following statements occur.
\begin{itemize}
\item[1.] If for each $\{S_1, \dots , S_{i+1}\}\subseteq \mathcal{E}(\mathcal{H})$ with $S_{i+1}\subseteq \bigcup_{\ell =1}^i S_\ell$ and $|\bigcup_{\ell=1}^iS_\ell |=j$ we have $S_k\nsubseteq \bigcup_{1\leq \ell \leq i+1,\ell\neq k}S_\ell$ for all $1\leq k \leq i$, then
$$|\{\mathcal{S}\ | \ \mathcal{S} \textrm{ is a self semi-induced matching  in  } \mathcal{H} \textrm{ of type }(i,j)\}| \leq  \beta_{i,j}(R/I(\mathcal{H})).$$
\item[2.] If for each $\{S_1, \dots , S_i\}\subseteq \mathcal{E}(\mathcal{H})$ with $|\bigcup_{\ell=1}^iS_\ell |=j$ we have $S_1\nsubseteq \bigcup_{\ell=2}^iS_\ell$, then
$$\beta_{i,j}(R/I(\mathcal{H}))\leq |\{\mathcal{S}\ | \ \mathcal{S} \textrm{ is a self-contained semi-induced matching  in  } \mathcal{H} \textrm{ of type }(i,j)\}|.$$
\end{itemize}
\end{theorem}
\begin{proof}
The results immediately follow from Lemmas \ref{Lemma2.2} and \ref{3.7.1}.
\end{proof}

The following corollary shows that Theorem \ref{3.7} generalizes some results in \cite{Ha+Vantuyl}, \cite{Katzman} and  \cite{M+V}.
\begin{corollary}(Compare \cite[Theorem 6.5]{Ha+Vantuyl}, \cite[Lemma 2.2]{Katzman} and \cite[Corollary 3.9]{M+V}.)\label{C3.7}
Suppose that $i$ and $j$ are  integers and
$$t=\max\{|S| \ | \ S\in \mathcal{E}(\mathcal{H})\}.$$ Then
\begin{itemize}
\item[1.] $\beta_{i,ti}(R/I(\mathcal{H}))$ is the number of induced matchings of $\mathcal{H}$ of size $i$ consisting of $t$-sets;
\item[2.] $\mathrm{reg}(R/I(\mathcal{H}))\geq (t-1)a_{\mathcal{H},t}$; and
\item[3.] if for all $i\geq e_{\mathcal{H}}$  the assumption of Part 1 of Lemma \ref{Lemma2.2} holds, then $\mathrm{pd}(R/I(\mathcal{H}))\leq e_{\mathcal{H}}$.
\end{itemize}
\end{corollary}
\begin{proof}
\begin{itemize}
\item[1.]  Note that if $\mathcal{S}=\{S_{\ell_1}, \dots , S_{\ell_i}\}$ is a family of edges of $\mathcal{H}$ with $ti=|\bigcup_{t =1}^i S_{\ell_t} |$, then clearly $\mathcal{S}$ is a matching. So, for such an $\mathcal{S}$, being  an induced matching, a semi-induced matching, a self semi-induced matching and a self-contained semi-induced matching are equivalent. Therefore, the result immediately follows from Lemma \ref{3.7.1} and Part 3 of Lemma \ref{Lemma2.2}.
\item[2.] can be gained from Part 1.
\item[3.]  immediately follows from Part 2 of Theorem \ref{3.7}.
\end{itemize}
\end{proof}

The following result is another main result of this section which is a generalization of Proposition 2.5 in \cite{Katzman}.

\begin{theorem}(Compare \cite[Proposition 2.5]{Katzman}.)\label{T3.7}
\begin{itemize}
\item[1.] If there is an induced matching or a self semi-induced matching of $\mathcal{H}$ of type $(i,j)$, then $\beta_{i,j}(R/I(\mathcal{H}))\neq 0$.
\item[2.]  If there is a self ordered set of edges of $\mathcal{H}$ of type $(i,j)$, then $\beta_{i,j}(R/I(\mathcal{H}))\neq 0$.
\item[3.]  $\max\{b_{\mathcal{H}},c_\mathcal{H}\}\leq \mathrm{pd}(R/I(\mathcal{H}))$ and $\max\{b'_{\mathcal{H}},c'_\mathcal{H}\}\leq \mathrm{reg}(R/I(\mathcal{H}))$.
\end{itemize}
\end{theorem}
\begin{proof}
\begin{itemize}
\item[1.] We use the Taylor resolution of $R/I$. Suppose that $\mathcal{S}=\{S_{\ell_1}, \dots , S_{\ell_i}\}$ is an induced matching or a self semi-induced matching in $\mathcal{H}$ of type $(i,j)$. Then  by Remarks \ref{remarks2.2}(1), $\overline{e_{\ell_1, \dots , \ell_i}}\in (\mathrm{Ker}\overline{\partial}_i )_j\setminus (\mathrm{Im}\overline{\partial}_{i+1})_j$. So, $\overline{e_{\ell_1, \dots ,\ell_i}}+ \mathrm{Im}\overline{\partial}_{i+1}$ is a non-zero element in $(\mathrm{Ker}\overline{\partial}_i / \mathrm{Im}\overline{\partial}_{i+1})_j$. Hence, $\beta_{i,j}(R/I)\neq 0$ as required.
\item[2.] Suppose that $\mathcal{S}=\{S_{\ell_1}, \dots , S_{\ell_i}\}$ is a self ordered set of edges in $\mathcal{H}$ of type $(i,j)$.  We use the Lyubeznik resolution of $R/I$ with ordering
$S_{\ell_1}, \dots, S_{\ell_i}, \mathcal{E}(\mathcal{H})\setminus{\mathcal{S}}$ on edges of $\mathcal{H}$. Then  by Parts 2 and 4 of Remarks \ref{remarks2.2}, $\overline{e_{\ell_1, \dots ,\ell_i}}\in (\mathrm{Ker}\overline{\sigma}_i)_j \setminus (\mathrm{Im}\overline{\sigma}_{i+1})_j$. So, $\overline{e_{\ell_1, \dots ,\ell_i}}+ \mathrm{Im}\overline{\sigma}_{i+1}$ is a non-zero element in $(\mathrm{Ker}\overline{\sigma}_i / \mathrm{Im}\overline{\sigma}_{i+1})_j$. Hence, $\beta_{i,j}(R/I)\neq 0$ as required.
\item[3.] immediately follows from Parts 1 and 2.
\end{itemize}
\end{proof}

\begin{example}
\begin{itemize}
\item[1.] Let $G$ be a fan graph with $V(G)=\{z,x_1, x_2, \dots, x_n\}$ and $E(G)=\bigcup_{i=1}^n\{zx_i\} \cup \bigcup_{i=1}^{n-1}\{x_i x_{i+1}\}$. One can easily check that
$$\{zx_i \ | \ 1\leq i \leq n\},$$ is a self ordered set of edges in $G$ of type $(n,n+1)$ with any order on the edges. So in view of Part 3 of Theorem \ref{T3.7}, we have
$\mathrm{pd}(K[z,x_1, \dots , x_n]/I(G))\geq n$.

\item[2.] Let $\mathcal{H}$ be a simple hypergraph with $m$ edges in which every edge has a free vertex (that is for each edge $S$ of $\mathcal{H}$, there is a vertex just belonging to $S$). Then clearly $\mathcal{E}(\mathcal{H})$  is the maximal self-contained semi-induced matching and self semi-induced matching of $\mathcal{H}$.  So $b_\mathcal{H}=e_\mathcal{H}=m$.  Hence, in view of Part 3 of Theorem \ref{T3.7} and Part 3 of Corollary \ref{C3.7}, $\mathrm{pd}(R/I(\mathcal{H}))=m$. In particular, when $G$ is disjoint union of star graphs with $m$ edges, then $\mathrm{pd}(R/I(G))=m$. (Of course this also can be immediately followed from th Taylor resolution, because free vertices make it minimal.)
\end{itemize}
\end{example}

In the light of Part 3 of Proposition \ref{proposition1.3}, the following theorem is a generalization of Theorem 3.1 in \cite{Kimura} and introduces some combinatorial lower bounds for $\mathrm{pd}( R/I(\mathcal{H}))$ and $\mathrm{reg}( R/I(\mathcal{H}))$.

\begin{theorem}(Compare \cite[Theorem 3.1]{Kimura}.)\label{1.9}
Assume that there exists a self semi-disjoint set of edges in $\mathcal{H}$ of type $(i,j)$. Then
$\beta_{i,j}(R/I(\mathcal{H}))\neq 0$. In particular
$$\mathrm{pd}( R/I(\mathcal{H}))\geq d_{2,\mathcal{H}},$$
 and
$$\mathrm{reg}( R/I(\mathcal{H}))\geq d'_{2,\mathcal{H}}.$$
\end{theorem}
\begin{proof}
In the light of Part 2 of Remarks \ref{1.3}, we may assume that $\mathcal{S}$ is a self semi-disjoint set in $\mathcal{H}$ of type $(i,j)$ and $V(\mathcal{H})=\bigcup_{S\in \mathcal{S}}S$. Then in view of Condition (ii) of Definitions \ref{1.2}(6), one can take a semi-induced matching $\mathcal{S}_0$ contained in  $\mathcal{S}$ with the desired condition. We put the ordering $(\mathcal{S}\setminus \mathcal{S}_0), (\mathcal{E}(\mathcal{H})\setminus \mathcal{S}), \mathcal{S}_0$ 
on the edges of $\mathcal{H}$. Consider the symbol $\sigma$   associated to $(\mathcal{S}\setminus \mathcal{S}_0), \mathcal{S}_0$. By Part 6 of Remarks \ref{remarks2.2}, $\sigma$ is an L-admissible symbol. Now, by means of Part 2 of Remarks \ref{remarks2.2} and Condition (i) in Definitions \ref{1.2}(6), it is enough to prove that $\sigma$ is a maximal L-admissible symbol. Suppose in contrary that there exists an edge $E\in \mathcal{E}(\mathcal{H})\setminus \mathcal{S}$ such that the symbol  $\tau$ associated to  
$(\mathcal{S}\setminus \mathcal{S}_0), E, \mathcal{S}_0$ 
is also L-admissible. Then since $\mathcal{S}_0$ is a semi-induced matching, $E\nsubseteq \bigcup_{S\in \mathcal{S}_0} S$. Hence there exists $x\in E\setminus \bigcup_{S\in \mathcal{S}_0}S$. Therefore, since $E\subseteq \bigcup _{S\in \mathcal{S}}S$, there exists $S\in \mathcal{S}\setminus \mathcal{S}_0$ such that $x\in S\setminus \bigcup _{S\in \mathcal{S}_0}S$. Now, in view of Condition (ii) of Definitions \ref{1.2}(6), there exists $S_0\in \mathcal{S}_0$ such that $|S\setminus S_0|=1$. Hence $S\setminus S_0=\{x\}$. This shows that $S\subseteq (E\cup S_0)$, since $x\in E$. Therefore $\tau$ can not be L-admissible. So $\sigma$ is a maximal L-admissible symbol as desired. The last assertion immediately follows from the first statement.
\end{proof}

\section{Some characterization for projective dimension and regularity of edge ideal of triangulated hypergraphs}
As we promised in the introduction, in this section we are willing to concentrate on a special class of triangulated hypergraphs and to characterize algebraic invariants of their edge ideals. Note that the concept of triangulated hypergraphs is a natural generalization of the concept of chordal graphs which firstly introduced in \cite{Ha+Vantuyl}.
Hereafter we assume that $\mathcal{H}$ is a $d$-uniform hypergraph such that for every two distinct edges $S$ and $S'$ which has non-empty intersection, we have $|S\cap S'|=d-1$.  For our next main result we need to recall some definitions from \cite{Ha+Vantuyl}.

\begin{definitions} (Compare \cite[Definitions 3.1, 4.6, 5.3, 5.4 and 5.5]{Ha+Vantuyl}.)
\begin{itemize}
\item[1.] An edge $S$ in $\mathcal{H}$ is called a \textbf{splitting edge} of $\mathcal{H}$ if $I(\mathcal{H})=\langle x^S \rangle +I(\mathcal{H}\setminus S)$ is a splitting of $I(\mathcal{H})$. Recall that $I=U+V$ is called a \textbf{splitting of the monomial ideal} $I$ if $U$ and $V$ are two monomial ideals such that $G(I)$ is the disjoint union of $G(U)$ and $G(V)$ and there is a function
\begin{align*}
G(U\cap V) &\rightarrow G(U)\times G(V)\\
w &\mapsto (\phi (w), \psi(w))
\end{align*}
satisfying the following properties:
\begin{itemize}
\item[$\bullet$] $w=\mathrm{lcm}(\phi (w), \psi(w))$ for all $w\in G(U\cap V)$;
\item[$\bullet$] for every subset $G'\subseteq G(U\cap V)$, both $\mathrm{lcm}\phi(G')$ and $\mathrm{lcm}\psi(G')$ strictly divide $\mathrm{lcm}G'$.
\end{itemize}
(see \cite{E+K}).
\item[2.] Let $S$ be an edge of $\mathcal{H}$. Then we set
$$N(S)=\bigcup_{E\in \mathcal{E}(\mathcal{H}), E\cap S\neq \emptyset}E\setminus S.$$
\item[3.] For each vertex $x$ in $\mathcal{H}$ we set
$$N(x)= \{y\in V(\mathcal{H}) \ | \ \textrm{ there exists } S\in \mathcal{E}(\mathcal{H}) \textrm{ with } \{x,y\}\subseteq S\},$$
and
$$N[x]=\{x\}\cup N(x).$$
\item[4.] A vertex $x$ in a $d$-uniform hypergraph $\mathcal{H}$ is called a \textbf{simplicial vertex} if every $d$-subset of $N[x]$ is an edge of $\mathcal{H}$.
\item[5.] $\mathcal{H}$ is called \textbf{triangulated} if every induced subhypergraph of $\mathcal{H}$ has a simplicial vertex.
\end{itemize}
\end{definitions}

\begin{remarks}\label{remarks1.10}
\begin{itemize}
\item[1.] Suppose that $x$ is a simplicial vertex and $S$ is an edge of $\mathcal{H}$ containing $x$. Set $N(S)=\{z_1, \dots, z_t\}$, where $z_1, \dots , z_t$ are pairwise distinct. Then there exist distinct edges  $S_1, \dots , S_t$ of $\mathcal{H}$ such that $S_\ell\setminus S=\{z_\ell\}$ and $x\notin S_\ell$ for all $1\leq \ell \leq t$. Actually, for each $1\leq \ell \leq t$ we can choose an edge $S_\ell$ of $\mathcal{H}$ with $z_\ell \in S_\ell$. Since $z_\ell\notin S$ and $|S_\ell \cap S|=d-1$, we have $S_\ell\setminus S=\{z_\ell\}$. If $x\in S_\ell$, then there is a vertex $y_\ell\neq x$ which belongs to $S\setminus S_\ell$. Since $x$ is a simplicial vertex, $(S_\ell\setminus \{x\})\cup \{y_\ell\}$ is an edge of $\mathcal{H}$ containing $z_\ell$ which doesn't contain $x$. So, we may replace $S_\ell$ with $(S_\ell\setminus \{x\})\cup \{y_\ell\}$.
\item[2.] In view of proof of Lemma 5.7 in \cite{Ha+Vantuyl} and the assumption $|S\cap S'|=d-1$ for every two distinct edges $S$ and $S'$ of $\mathcal{H}$, if $\mathcal{H}$ is triangulated, $x$ is a simplicial vertex of $\mathcal{H}$, $S=\{x, x_2, \dots , x_d\}$ is an edge containing $x$ and $N(S)=\{z_1, \dots, z_t\}$,  then $S$ is a splitting edge and $\mathcal{H}_1$ and $\mathcal{H}_2$ are triangulated hypergraphs, where $\mathcal{H}_1=\mathcal{H}\setminus S$ and $\mathcal{H}_2$ is the induced subhypergraph on $V(\mathcal{H})\setminus \{x,x_2, \dots, x_d, z_1, \dots , z_t\}$.
\end{itemize}
\end{remarks}

The following two lemmas, which is needed for our next main result, may be valuable in turn.
\begin{lemma}\label{lemma1.9}
Assume that  $x$ is a simplicial vertex in $\mathcal{H}$ and $S$ is an edge of $\mathcal{H}$ containing $x$. Then if $\mathcal{S}$ is an induced matching in $\mathcal{H}\setminus S$, then $\mathcal{S}$ is an induced matching in $\mathcal{H}$. In particular, if $\mathcal{S}$ is a self disjoint set in $\mathcal{H}\setminus S$, then $\mathcal{S}$ is a self disjoint set in $\mathcal{H}$.
\end{lemma}
\begin{proof}
Set $\mathcal{S}=\{S_1, \dots , S_i\}$. In order to establish the first assertion, it is enough to prove that $S\nsubseteq \bigcup _{\ell=1}^iS_\ell$. Suppose in contrary that $S\subseteq \bigcup_{\ell=1}^i S_{\ell}$. Then we may assume that $x\in S_1$. Since $S_1\neq S$,  there exists a vertex $y\in S_1\setminus S$. Since $x$ is a simplicial vertex, $y\in N(x)$ and $S\setminus \{x\}\subseteq N(x)$, we have  $(S\setminus \{x\})\cup \{y\}\in \mathcal{E}(\mathcal{H}\setminus S)$ and $(S\setminus \{x\})\cup \{y\}\subseteq \bigcup_{\ell=1}^iS_{\ell}$. Therefore since $\mathcal{S}$ is an induced matching in $\mathcal{H}\setminus S$, we may assume that $S_2=(S\setminus \{x\})\cup \{y\}$. But now  we have $y\in S_1\cap S_2$ which is a contradiction with the fact that $\mathcal{S}$ is a matching. The last assertion immediately follows from the first statement.
\end{proof}

\begin{lemma}\label{lemma1.10}
With the notation as in Remarks \ref{remarks1.10}, if $\mathcal{S}'$ is a self disjoint set in $\mathcal{H}_2$ of type $(i-1-t, j-d-t)$, then $\mathcal{S}=\mathcal{S}'\cup \{S, S_1, \dots , S_t\}$ is a self disjoint set in $\mathcal{H}$ of type $(i,j)$.
\end{lemma}
\begin{proof}
We first prove that $\mathcal{S}$ satisfies Condition (i) in Definitions \ref{1.2}(6). Note that 
$$\bigcup_{S'\in \mathcal{S}'}S'\subseteq V(\mathcal{H})\setminus \{x,x_2, \dots, x_d, z_1, \dots , z_t\}=V(\mathcal{H})\setminus (S\cup (\bigcup_{k=1}^tS_k)).$$
Also note that $x\in S \setminus (\bigcup_{k=1}^tS_k)$ and $z_k\in S_k\setminus (S\cup (\bigcup_{k'\neq k}S_{k'}))$. Then assertion follows by the fact that $\mathcal{S}'$ is a self disjoint set in $\mathcal{H}_2$.

Next we prove that $\mathcal{S}$ satisfies Condition (ii) in Definitions \ref{1.2}(6). Since $\mathcal{S}'$ is a self disjoint set in $\mathcal{H}_2$, there exists an induced matching $\mathcal{S}'_0 \subseteq \mathcal{S}'$ with the property mentioned in Condition (ii) of Definitions \ref{1.2}(6). We will show that
$\mathcal{S}_0=\mathcal{S}'_0 \cup \{S\}$ is a desired induced matching. We first show $\mathcal{S}_0$ is an induced matching. It is clear that $\mathcal{S}_0$ is a matching. Also suppose in contrary that there exists an edge $E$ of $\mathcal{H}$ which is contained in $S\cup (\bigcup _{S'\in \mathcal{S}'_0}S')$ and $E\neq S$ and $E\notin \mathcal{S}'_0$. Since $\mathcal{S}'_0$ is an induced matching, $E\cap S\neq \emptyset$. Hence there is a vertex $z$ such that $E\setminus S=\{z\}$ because $|E\cap S|=d-1$. Since $E\subseteq S\cup (\bigcup_{S'\in \mathcal{S}'_0}S')$, we have $z\in \bigcup _{S'\in \mathcal{S}'_0}S'$. Although $E\setminus S=\{z\}$ implies $z\in N(S)=\{z_1, \dots , z_t\}$, it contradicts to $z\in \bigcup _{S'\in \mathcal{S}'_0}S'\subseteq V(\mathcal{H})\setminus \{x,x_2, \dots, x_d, z_1, \dots , z_t\}$. Thus $\mathcal{S}_0$ is an induced matching. Other properties follows from the fact that $\mathcal{S}'$ is a self disjoint set in $\mathcal{H}_2$ and for all $1\leq k \leq t$, $S_k \setminus S=\{z_k\}$.

So $\mathcal{S}$ is a self disjoint set in $\mathcal{H}$. It is clear that it is of type $(i,j)$ as desired.
\end{proof}

Now, we are ready to establish our main result of this section which is a generalization of Theorem 4.1 in \cite{Kimura}.
\begin{theorem} \label{1.11}
Assume that $\mathcal{H}$ is a $d$-uniform triangulated  hypergraph such that for every distinct non-disjoint edges $S$ and $S'$, we have $|S\cap S'| = d-1$.
\begin{itemize}
\item[1.] $\beta_{i,j}(R/I(\mathcal{H}))\neq 0$ if and only if $\mathcal{H}$ contains a self disjoint set of type $(i,j)$.
\item[2.] $\mathrm{pd}(R/I(\mathcal{H}))=d_{1,\mathcal{H}}=d_{2,\mathcal{H}}$.
\item[3.] $\mathrm{reg}(R/I(\mathcal{H}))=d'_{1,\mathcal{H}}=d'_{2,\mathcal{H}}$.
\end{itemize}
\end{theorem}
\begin{proof}
\begin{itemize}
\item[1.] By using Theorem \ref{1.9}, it is enough to prove the \textit{only if} part. To this end, in the light of Remarks \ref{1.3}(1), we may assume that $|V(\mathcal{H})|=j$. We use induction on $|\mathcal{E}(\mathcal{H})|=m$. If $m=0$, then $I(\mathcal{H})=0$ and hence the only non-zero graded Betti number is $\beta_{0,0}$. So  $\mathcal{S}=\emptyset$ is the desired self disjoint set. If $m=1$, then $I(\mathcal{H})=\langle x^S \rangle$ where $S$ is the only edge of $\mathcal{H}$ and so the only non-zero graded Betti numbers are $\beta_{0,0}$ and $\beta_{1,|S|}$. Since $\mathcal{S}=\emptyset$ and $\mathcal{S}=\{S\}$ are self disjoint sets in $\mathcal{H}$, we have done. Now assume inductively that the result has been proved for smaller values of $m$. Since $\mathcal{H}$ is triangulated, it has a simplicial vertex, say $x$. Let $S=\{x, x_2, \dots , x_d\}$ be an edge containing $x$ in $\mathcal{H}$. Set $N(S)=\{z_1, \dots , z_t\}$. Then in view of Remarks \ref{remarks1.10}(2), $S$ is a splitting edge and $\mathcal{H}_1$ and $\mathcal{H}_2$ are triangulated hypergraphs, where $\mathcal{H}_1=\mathcal{H}\setminus S$ and $\mathcal{H}_2$ is the induced subhypergraph on $V(\mathcal{H})\setminus \{x,x_2, \dots, x_d, z_1, \dots , z_t\}$. Moreover, by using Remarks \ref{remarks1.10}(1),  there exists $S_1, \dots , S_t\in \mathcal{E}(\mathcal{H})$ such that $S_\ell\setminus S=\{z_\ell\}$  and $x\notin S_\ell$ for all $1\leq \ell \leq t$. By means of Theorem 5.8 in \cite{Ha+Vantuyl}, we have the recursive formula
\begin{equation}\label{equation1}
\beta_{i,j}(R/I(\mathcal{H}))=\beta_{i,j}(R/I(\mathcal{H}_1) )+\sum_{\ell =0}^{i-1}
\left( \begin{array}{c}
t \\
\ell
\end{array}   \right) \beta_{i-1-\ell, j-d-\ell}(R/I(\mathcal{H}_2)).
\end{equation}
Since $|V(\mathcal{H}_2)|=j-d-t$, by Remarks \ref{1.3}(3), $ \beta_{i-1-\ell, j-d-\ell}(R/I(\mathcal{H}_2))=0$ except when $\ell\geq t$. Hence (\ref{equation1}) implies that
\begin{equation}\label{equation2}
\beta_{i,j}(R/I(\mathcal{H}))=\beta_{i,j}(R/I(\mathcal{H}_1) )+\beta_{i-1-t, j-d-t}(R/I(\mathcal{H}_2)).
\end{equation}
Now, since $\beta_{i,j}(R/I(\mathcal{H}))\neq 0$, at least one of the summands in the right side of (\ref{equation2}) is non-zero. Thus there are two cases. We solve the problem in each case as follows.
\begin{itemize}

\item[Case 1.] $\beta_{i,j}(R/I(\mathcal{H}_1))\neq 0$.  So, by inductive hypothesis $\mathcal{H}_1$ has a self disjoint set $\mathcal{S}$ of type $(i,j)$.  Now the result follows from Lemma \ref{lemma1.9} in this case.

\item[Case 2.] $\beta_{i-1-t, j-d-t}(R/I(\mathcal{H}_2))\neq 0$.
So by inductive hypothesis, there is a self disjoint set $\mathcal{S}'=\{S_{\ell_1}, \dots , S_{\ell_{i-1-t}}\}$ in $\mathcal{H}_2$ of type $(i-1-t,j-d-t)$. By means of Lemma \ref{lemma1.10}, $\mathcal{S}=\{S_{\ell_1}, \dots , S_{\ell_{i-1-t}}, S, S_1, \dots , S_t\}$ is a self disjoint set in $\mathcal{H}$ of type $(i,j)$ as desired.
\end{itemize}
 The above cases  complete the proof.

\item[2, 3.] can be implied from Part 1, Theorem \ref{1.9} and Remarks \ref{remarks1.3}(2).
\end{itemize}
\end{proof}

As an immediate consequence of Theorem \ref{1.11} and Proposition \ref{proposition1.3}(3), we regain some results in \cite{Kimura} and \cite{Zheng} as follows.
\begin{corollary} (See \cite[Theorem 4.1]{Kimura} and  \cite{Zheng}.)
Assume that $G$ is a chordal graph.
Then
$$\mathrm{pd}(R/I(G))=d_G,$$
and
$$\mathrm{reg}(R/I(G))=d'_G.$$
\end{corollary}

The following example presents some classes of hypergraphs satisfying the assumptions of Theorem \ref{1.11}.
\begin{example}
Assume that $\mathcal{H}$ is a hypergraph on the vertex set $V(\mathcal{H})=\{x_1, \dots , x_{d+1}\}$ whose edges are all $d$-subsets of $V(\mathcal{H})$. Also assume that $\mathcal{H}'$ is a star hypergraph, i.e. $\mathcal{H}'$ is a  hypergraph with $V(\mathcal{H}')=\{z_1, \dots , z_{d-1}, x_1, \dots , x_n\}$ and
$$\mathcal{E}(\mathcal{H}')=\{\{z_1, \dots , z_{d-1}, x_i\} \ | \ 1\leq i \leq n\}.$$
Then one can easily check that $\mathcal{H}$ and $\mathcal{H}'$ are the hypergraphs satisfying the assumptions of Theorem \ref{1.11}. Also one can check that $\{\{x_1, \dots, x_d\}, \{x_2, \dots , x_{d+1}\}\}$ is a self disjoint set in $\mathcal{H}$ of the maximum size and for each self disjoint set in $\mathcal{H}$ of type $(i,j)$, $j-i=d-1$.  Moreover, $\mathcal{E}(\mathcal{H}')$  is a self disjoint set in $\mathcal{H}'$ of the maximum size and for each self disjoint set in $\mathcal{H}'$ of type $(i,j)$, $j-i=d-1$. So, Theorem \ref{1.11} shows that
$$\mathrm{pd}(R/I(\mathcal{H}))=2, \mathrm{reg}(R/I(\mathcal{H}))=d-1,$$
and
$$\mathrm{pd}(R/I(\mathcal{H}'))=n, \mathrm{reg}(R/I(\mathcal{H}'))=d-1.$$
\end{example}

We end this paper by the following remark about triangulated hypergraphs.
\begin{remark}
Note that if $\mathcal{H}$ is a $d$-uniform properly connected  triangulated  hypergraph, then in view of Theorem 6.8 in \cite{Ha+Vantuyl} and the paragraph before Remarks \ref{remarks1.3}, we have
$$\mathrm{reg}(R/I(\mathcal{H}))= (d-1)a_{\mathcal{H}}.$$
Hence, if $\mathcal{H}$ is a $d$-uniform triangulated  hypergraph such that for every distinct non-disjoint edges $S$ and $S'$, we have $|S\cap S'| = d-1$, then Part 3 of Theorem \ref{1.11} implies that
$$d'_{1,\mathcal{H}}=d'_{2,\mathcal{H}}=(d-1)a_{\mathcal{H}}.$$
In particular, if $G$ is a chordal graph, then $d'_G=a_G$.
Of course note that, in view of definitions of $d'_{1,\mathcal{H}}$ and $a_{\mathcal{H}}$, in a $d$-uniform hypergraph we always have  
$$d'_{1,\mathcal{H}}=(d-1)a_{\mathcal{H}}.$$
\end{remark}

\textbf{Acknowledgments.}
 The authors would like to thank the anonymous reviewer whose comments and remarks improved the presentation of the paper.

\providecommand{\bysame}{\leavevmode\hbox
to3em{\hrulefill}\thinspace}


\begin{thebibliography}{10}

\bibitem{Barile} {\sc Barile, M.} {\it  On ideals whose radical is a monomial ideal.} Comm. Algebra 33 (2005), no. 12, 4479--4490.

\bibitem{Berge} {\sc Berge, C.}  Graphs and hypergraphs. Translated from the French by Edward Minieka. North-Holland Mathematical Library, Vol. 6. North-Holland Publishing Co., Amsterdam-London; American Elsevier Publishing Co., Inc., New York, 1973. {\rm xiv}+528 pp.

\bibitem{E+K} {\sc Eliahou, S.; Kervaire, M.} {\it Minimal resolutions of some monomial ideals.} J. Algebra 129 (1990), no. 1, 1--25.

\bibitem{Ha+Vantuyl} {\sc Ha, H. T.; Van Tuyl, A.} {\it Monomial ideals, edge ideals of hypergraphs, and their graded Betti numbers.} J. Algebraic Combin. 27 (2008), no. 2, 215--245.

\bibitem{Hoch} {\sc Hochster, M.} Cohen-Macaulay rings, combinatorics, and simplicial complexes. Ring theory, II (Proc. Second Conf., Univ. Oklahoma, Norman, Okla., 1975), pp. 171--223. Lecture Notes in Pure and Appl. Math., Vol. 26, Dekker, New York, 1977.

\bibitem{Katzman} {\sc Katzman, M.} {\it Characteristic-independence of Betti numbers of graph ideals.} J. Combin. Theory Ser. A 113 (2006), no. 3, 435--454.

\bibitem{FS} {\sc Khosh-Ahang, F.; Moradi, S.}  {\it Regularity and projective dimension of the edge ideal of $C_5$-free vertex decomposable graphs.} Proc. Amer. Math. Soc. 142 (2014), no. 5, 1567--1576.

\bibitem{Kimura} {\sc Kimura, K.} {\it  Non-vanishingness of Betti numbers of edge ideals.} Harmony of Gr$\T{\ddot{o}}$bner bases and the modern industrial society, 153--168, World Sci. Publ., Hackensack, NJ, 2012.

\bibitem{Lyubeznik} {\sc Lyubeznik, G.}  {\it A new explicit finite free resolution of ideals generated by monomials in an $R$-sequence.} J. Pure Appl. Algebra 51 (1988), no. 1-2, 193--195.

\bibitem{F+S} {\sc Moradi, S.; Khosh-Ahang, F.} {\it  Matchings in hypergraphs and Castelnuovo-Mumford regularity}. Publ. Math. Debrecen 91 (2017), no. 3-4, 427--439.

\bibitem{M+V} {\sc Morey, S.; Villarreal, R. H.} {\it  Edge ideals: algebraic and combinatorial properties.} Progress in commutative algebra 1, 85--126, de Gruyter, Berlin, 2012.

\bibitem{Villarreal} {\sc Villarreal, R. H.} {\it Cohen-Macaulay graphs.} Manuscripta Math. 66 (1990) 277--293.

\bibitem{Zheng} {\sc Zheng, X.} {\it  Resolutions of facet ideals.} Comm. Algebra 32 (2004), no. 6, 2301--2324.





\end{thebibliography}
\end{document}